\documentclass{amsart}
\usepackage[utf8x]{inputenc}

\usepackage{booktabs} % for much better looking tables
\usepackage{array} % for better arrays (eg matrices) in maths
\usepackage{paralist} % very flexible & customisable lists (eg. enumerate/itemize, etc.)
\usepackage{verbatim} % adds environment for commenting out blocks of text & for better verbatim
\usepackage{subfig} % make it possible to include more than one captioned figure/table in a single float
\usepackage{amsfonts}
\usepackage{amsthm}
\usepackage{amsmath}
\usepackage{amssymb}
\usepackage[all, cmtip]{xy} % nicer arrowheads in xymatrix
\usepackage{url}
\usepackage{color}
\usepackage{todonotes} % allows notes on work in progress
\usepackage{hyperref} % hyperlinks

\newcommand{\mc}{\mathcal}
\newcommand{\mb}{\mathbb}

\newcommand{\tn}{\textnormal}
\newcommand{\exact}[3]{#1 \rightarrowtail #2 \twoheadrightarrow #3}
\newdir{ >}{{}*!/-5pt/\dir{>}}

\newtheorem{thm}{Theorem}[section] 
\newtheorem{prop}[thm]{Proposition} 
\newtheorem{lem}[thm]{Lemma}
\newtheorem{cor}[thm]{Corollary} 
 
\newtheorem{thmdef}[thm]{Theorem / Definition}

\theoremstyle{definition}
\newtheorem{defn}[thm]{Definition}
\newtheorem*{theorem}{Theorem}

\newtheorem{ex}[thm]{Example}

\newtheorem*{ack}{Acknowledgements}

\begin{document}

\title[Fundamental theorems in algebraic $K$-theory]{Algebraic proofs of some fundamental theorems in algebraic $K$-theory} 

%    Remove any unused author tags.

%    author one information
\author{Tom Harris}
\address{Mathematical Sciences, University of Southampton, SO17 1BJ, UK}
\email{tkmharris@gmail.com}
\urladdr{http://www.southampton.ac.uk/~tkh1v07/}
\thanks{}



\date{November $20^{\tn{th}}$, 2013}

\dedicatory{}

\begin{abstract}
We present news proofs of the additivity, resolution and cofinality theorems for the algebraic
$K$-theory of exact categories. These proofs are entirely algebraic, based on Grayson's presentation
of higher algebraic $K$-groups via binary complexes.
\end{abstract}

\maketitle

\section*{Introduction}

The beautiful and relatively young discipline of algebraic $K$-theory has seen 
tremendous development and far-reaching applications in many other mathematical disciplines
over the last decades. This paper makes a contribution to a project (begun in \cite{Gray1} and \cite{Gray5})
reformulating its foundations.\\

The algebraic $K$-theory of an exact category was first described by Segal and Waldhausen, obtained
by modifying Segal's construction of the $K$-theory of a symmetric monoidal category. 
Quillen's alternative $Q$-construction gives a very powerful tool
for proving fundamental theorems in algebraic $K$-theory, which he exploited to prove
the additivity, resolution, d\'{e}vissage and localisation theorems \cite{Quil1}. 
Waldhausen's later work on the $S_{\cdot}$-construction, in particular 
his version of the additivity theorem, made simpler proofs of the theorems cited above and the cofinality theorem
possible \cite{Staffeldt}. Common in all of these approaches
is the use of some non-trivial content from homotopy theory. \\

Grayson \cite{Gray1} recently gave the first
presentation of the higher algebraic $K$-groups of an an exact category by generators and relations; 
we take this presentation as our definition of $K_{n}\mc{N}$. The object of this paper is to present 
new completely algebraic proofs of the additivity, resolution and cofinality
theorems in higher algebraic $K$-theory of exact categories. \\

We assume throughout that the reader is familiar with exact categories. They are first systematically
defined in \cite{Quil1}, a very nice exposition is \cite{Buh}.
In section \ref{background} we review the necessary details of Grayson's presentation and present 
a new proof of the additivity theorem, in its form concerning so-called extension categories.

\begin{theorem}[Additivity]
 Let $\mc{B}$ be an exact category, with exact subcategories $\mc{A}$ and $\mc{C}$ closed under 
 extensions in $\mc{B}$. Let $\mc{E}(\mc{A},\mc{B},\mc{C})$ denote the associated extension category. 
 Then $K_{n}\mc{E}(\mc{A},\mc{B},\mc{C}) \cong K_{n}\mc{A} \times K_{n}\mc{C}$, for every $n \geq 0$.
\end{theorem}

Using Grayson's presentation, the proof of the additivity theorem is rather simple. 
In sections \ref{resolution} and \ref{cofinality} we present more involved proofs of the resolution and 
cofinality theorems.

\begin{theorem}[Resolution]
  Let $\mc{M}$ be an exact category and let $\mc{P}$ be a full, additive subcategory
 that is closed under extensions. Suppose also that:
\begin{enumerate}
 \item If $\exact{M'}{M}{M''}$ is an exact sequence in $\mc{M}$ with $M$ and $M''$ in $\mc{P}$ then
 $M'$ is in $\mc{P}$ as well.
 \item Given $j : M \twoheadrightarrow P$ in $\mc{M}$ with $P$ in $\mc{P}$ there exists 
 $j': P' \twoheadrightarrow P$ and $f : P' \rightarrow M$ in $\mc{M}$ with $P'$ in $\mc{P}$ such that 
 $jf = j'$.
 \item Every object of $\mc{M}$ has a finite resolution by objects of $\mc{P}$.
\end{enumerate}
 Then the inclusion functor $\mc{P} \hookrightarrow \mc{M}$ induces an isomorphism
 $K_{n}(\mc{P}) \cong K_{n}(\mc{M})$ for every $n \geq 0$.
\end{theorem}

\begin{theorem}[Cofinality]
 Let $\mc{M}$ be a cofinal exact subcategory of an exact category $\mc{N}$. Then the inclusion functor 
 $\mc{M} \hookrightarrow \mc{N}$ induces an injection $K_{0}\mc{M} \hookrightarrow K_{0}\mc{N}$ and isomorphisms
 $K_{n}\mc{M} \cong K_{n}\mc{N}$ for $n > 0$.
\end{theorem}

Grayson defines the $n^{\tn{th}}$
algebraic $K$-group of an exact category, denoted $K_{n}\mc{N}$, as a quotient group of the Grothendieck 
group of a certain related exact category $(B^{\tn{q}})^{n}\mc{N}$, whose objects are so-called 
\textit{acyclic binary multicomplexes} (see Definition\autoref{binary-multi}). 
Each of the theorems above makes some comparison between the $K$-groups of a pair of exact categories. 
These theorems all have well-known algebraic folk proofs for the Grothendieck group $K_{0}$ so
the general schema for our proofs is then as follows. First we verify that
the hypotheses on our exact categories of interest also hold for their associated categories of 
acyclic binary multicomplexes. Then we apply the algebraic $K_{0}$ proof to obtain a comparison between their 
Grothendieck groups. Finally we verify that the required comparison still holds when we pass to the quotients 
defining the higher algebraic $K$-groups.\\

The remaining theorems regarded as fundamental in the algebraic $K$-theory of exact categories are the 
d\'{e}vissage and localisation theorems. These theorems
concern abelian categories, say $\mc{A}$ and $\mc{B}$. While the associated categories 
$(B^{\tn{q}})^{n}\mc{A}$ and $(B^{\tn{q}})^{n}\mc{B}$ are still exact, they will no longer be 
abelian, so a strategy more sophisticated than the approach of this paper will be necessary to prove
these theorems in the context of Grayson's new definition of the higher algebraic $K$-groups.

\begin{ack} 
 The author thanks Daniel Grayson for an essential insight into the proof of the resolution theorem,
 and for the results of his paper \cite{Gray1}, which we draw upon extensively. Thanks are also due
 for his helpful comments on a late draft of this paper.\\
 The author also thanks his supervisor Bernhard K\"{o}ck for his careful 
 readings of successive drafts of this work, for suggesting numerous improvements to its content, 
 and for his encouragement and enthusiasm. 
\end{ack}

\section{Grayson's binary complex algebraic $K$-theory}
\label{background}

In this section we recall the definitions and main result of \cite{Gray1}. 
As a first application we give a simple new proof of the additivity theorem, as
previously proven in \cite{Quil1}, as well as \cite{McCarthy} and \cite{Gray4},
and whose version for the $S_{\cdot}$-construction is commonly considered to be
\textit{the} fundamental theorem in the algebraic $K$-theory of spaces. We shall work throughout 
with exact categories in the sense of Quillen \cite{Quil1}, that is, 
additive categories with a distinguished collection of short exact sequences that satisfies
a certain set of axioms.

\begin{defn}
\label{acyclic}
 A bounded \textit{acyclic complex}, or \textit{long exact sequence}, in an exact category $\mc{N}$ is a
 bounded chain complex $N$ whose differentials factor through short exact sequences of $\mc{N}$. That is, 
 the differentials $d_{k} : N_{k} \rightarrow  N_{k-1}$ factor as 
 $N_{k} \twoheadrightarrow Z_{k-1} \rightarrowtail N_{k-1}$ such that each
 $\exact{Z_{k}}{N_{k}}{Z_{k-1}}$ is a short exact sequence of $\mc{N}$.
\end{defn}

In an abelian category the long exact sequences defined above agree with the usual long exact sequences.
Care must be taken in the case of a general exact category, as the following example shows. 

\begin{ex}
\label{proj and free}
 Let $R$ be a ring with a finitely-generated, 
stably-free, non-free projective module $P$ (so $P \oplus R^{m} \cong R^{n}$ for some $m$ and $n$).
We have short exact sequences of $R$-modules
$0 \rightarrow P \stackrel{i}{\rightarrow} R^{n} \stackrel{p}{\rightarrow} R^{m} \rightarrow 0$
and $0 \rightarrow R^{m} \stackrel{j}{\rightarrow} R^{n} \stackrel{q}{\rightarrow} P \rightarrow 0$,
where $i,j,p$ and $q$ are the obvious inclusions and projections.
The sequence 
\[ 0 \rightarrow R^{m} \stackrel{j}{\rightarrow} R^{n} \stackrel{iq}{\rightarrow} 
R^{n} \stackrel{p}{\rightarrow} R^{n} \rightarrow 0 \]
is a chain complex in the exact category of finitely-generated free modules, $\mathbf{Free}(R)$, 
and it is exact as a sequence of 
$R$ modules, but, in the sense of the definition above, it is not long exact in $\mathbf{Free}(R)$.
\end{ex}

The category $C^{\tn{q}}\mc{N}$
\footnote{The $\tn{q}$ stands for ``quasi-isomorphic to the zero complex''. 
We will not make use of the notion of a quasi-isomorphism---the
$\tn{q}$ is a reminder that we are dealing only with the \textit{acyclic} complexes in $\mc{N}$.} 
of bounded acyclic complexes in an exact category $\mc{N}$ is itself an exact category (\cite{Gray1}, $\S 6$); 
a composable pair of chain maps 
is declared to be short exact if and only if each composable pair of term-wise morphisms is short exact in $\mc{N}$.
That is, a composition of chain maps 
$N' \stackrel{\phi}{\rightarrow} N \stackrel{\psi}{\rightarrow} N''$
is short exact if and only if every 
$N'_{k} \stackrel{\phi_{k}}{\rightarrow} N_{k} \stackrel{\psi_{k}}{\rightarrow} N''_{k}$
is short exact. A word of warning here: a morphism in $C^{\tn{q}}\mc{N}$ that has admissible epimorphisms of $\mc{N}$
as its term-wise morphisms is not necessarily an admissible epimorphism in $C^{\tn{q}}\mc{N}$. 

\begin{ex}
\label{complex-epi}
 Let $i$, $j$, $p$ and $q$ 
be the morphisms in Example\autoref{proj and free}, and note that $iq + jp = 1$.
The diagram below has exact rows, and is in fact a morphism of 
$C^{\tn{q}}\mathbf{Free}(R)$.
\[
 \xymatrix{
 R^{m} \ar@{ >->}[rr]^{\left[\begin{smallmatrix} j \\ 1 \end{smallmatrix}\right]} \ar@{->>}[d] & &
 R^{n} \oplus R^{m} \ar[rr]^{\left[\begin{smallmatrix} 1 & -j \\ -p & 1 \end{smallmatrix}\right]} 
 \ar@{->>}[d]_{\left[\begin{smallmatrix} -p & 1 \end{smallmatrix}\right]} & &
 R^{n} \oplus R^{m} \ar@{->>}[rr]^{\left[\begin{smallmatrix} p & 1 \end{smallmatrix}\right]} 
 \ar@{->>}[d]^{\left[\begin{smallmatrix} 0 & 1 \end{smallmatrix}\right]} & &
 R^{m} \ar@{->>}[d] \\
 0 \ar@{ >->}[rr] & & R^{m} \ar[rr]^{1} & & R^{m} \ar@{->>}[rr] & & 0.
 }
\]
Each vertical arrow is an admissible epimorphism of $\mathbf{Free}(R)$, but the diagram is not an admissible
epimorphism of $C^{\tn{q}}\mathbf{Free}(R)$---its kernel is the complex discussed in Example\autoref{proj and free}, which
is not acyclic in $\mathbf{Free}(R)$.
\end{ex}

\begin{defn}
\label{binary}
 A \textit{binary complex} in $\mc{N}$ is a chain complex with two independent differentials. 
 More precisely, a binary complex is a triple $(N,d,\tilde{d})$ such that $(N,d)$ and $(N, \tilde{d})$
 are chain complexes in $\mc{N}$. We call a binary chain complex \textit{acyclic} if each of the complexes
 $(N,d)$ and $(N, \tilde{d})$ is acyclic in $\mc{N}$. A morphism between binary complexes is a morphism 
 between the underlying graded objects that commutes with both differentials. A short exact sequence
 is a composable pair of such morphisms that is short exact term-wise.
\end{defn}

Since $C^{\tn{q}}\mc{N}$ is an exact category, the reader may easily check that the category $B^{\tn{q}}\mc{N}$ 
of bounded acyclic binary complexes in $\mc{N}$ is also an exact category. There is a diagonal functor 
$\Delta : C^{\tn{q}}\mc{N} \rightarrow B^{\tn{q}}\mc{N}$, sending $(N,d)$ to $(N,d,d)$.
A binary complex that is in the image of $\Delta$ is also called \textit{diagonal}.
The diagonal functor is split by the \textit{top} and \textit{bottom} functors 
$\top , \bot : B^{\tn{q}}\mc{N} \rightarrow C^{\tn{q}}\mc{N}$;
it is clear that $\Delta, \top$ and $\bot$ are all exact. \\

Taking the category of acyclic binary complexes behaves well with respect to subcategories
closed under extensions. If $\mc{M}$ is a full subcategory closed 
under extension in $\mc{N}$ (later just called an \textit{exact subcategory}), 
then $B^{\tn{q}}\mc{M}$ is a subcategory closed under extensions in $B^{\tn{q}}\mc{N}$.
It is important here that the binary complexes in $B^{\tn{q}}\mc{N}$ are bounded. Starting at the
final non-zero term, one argues by induction that the objects that the extension factors through 
are actually in $\mc{M}$, using the hypothesis on $\mc{M}$ and $\mc{N}$, and the $3 \times 3$ Lemma 
(\cite{Buh}, Corollary 3.6).\\

Since $B^{\tn{q}}\mc{N}$ is an exact category, we can iteratively define an exact category
$(B^{\tn{q}})^{n}\mc{N} = B^{\tn{q}}B^{\tn{q}} \cdots B^{\tn{q}}\mc{N}$ for each $n \geq 0$. 
The objects of this category are bounded acyclic binary complexes of bounded acyclic binary complexes ...
of objects of $\mc{N}$. Happily, this may be neatly unwrapped: it is obvious that the following is an equivalent
definition of $(B^{\tn{q}})^{n}\mc{N}$.

\begin{defn}
\label{binary-multi}
 The exact category $B^{\tn{q}}\mc{N}$ of bounded acyclic binary multicomplexes of dimension 
 $n$ in $\mc{N}$ is defined as follows. A \textit{bounded acyclic binary multicomplex of dimension $n$}
 is a bounded (i.e. only finitely many non-zero), $\mb{Z}^{n}$-graded collection of objects of 
 $\mc{N}$ together with a pair of acyclic differentials $d^{i}$ and $\tilde{d}^{i}$ in each direction 
 $1 \leq i \leq n$ such that, for $i \neq j$, 
 \[
 \begin{array}{rcl}
 d^{i}d^{j} &=& d^{j}d^{i} \\ 
 d^{i}\tilde{d}^{j} &=& \tilde{d}^{j}d^{i} \\ 
 \tilde{d}^{i}d^{j} &=& d^{j}\tilde{d}^{i} \\
 \tilde{d}^{i}\tilde{d}^{j} &=& \tilde{d}^{j}\tilde{d}^{i}.
 \end{array}
 \]
 In other words, any pair of differentials in different directions commute. A morphism 
 $\phi : N \rightarrow N'$ between such binary multicomplexes is a $\mb{Z}^{n}$-graded collection of 
 morphisms of $\mc{N}$ that commutes with all of the differentials of $N$ and $N'$. A short exact sequence 
 in $(B^{\tn{q}})^{n}\mc{N}$ is a composable pair of such morphisms that is short exact term-wise.
\end{defn}

In addition to $(B^{\tn{q}})^{n}\mc{N}$, for each $n \geq 1$ we have an exact category 
$C^{\tn{q}}(B^{\tn{q}})^{n-1}\mc{N}$ of bounded acyclic complexes of objects of $(B^{\tn{q}})^{n-1}\mc{N}$. 
For each $i$ with $1 \leq i \leq n$ there is a diagonal functor 
$\Delta_{i} : C^{\tn{q}}(B^{\tn{q}})^{n-1}\mc{N} \rightarrow (B^{\tn{q}})^{n}\mc{N}$ that consists of
`doubling up' the differential of the (non-binary) acyclic complex and regarding it as direction $j$
in the resulting acyclic binary multicomplex. Any object of $(B^{\tn{q}})^{n}\mc{N}$ that is in the image 
of one of these $\Delta_{i}$ is called \textit{diagonal}. The diagonal
binary multicomplexes are those that have $d^{i} = \tilde{d}^{i}$ for \textit{at least one} $i$.\\

We can now formulate Grayson's presentation of the algebraic $K$-theory groups of $\mc{N}$,
which we shall take to be their definition for the remainder of this paper.

\begin{thmdef}[\cite{Gray1}, Corollary 7.4]
\label{Grayson's def}
 For $\mc{N}$ an exact category and $n \geq 0$, the abelian group $K_{n}\mc{N}$ is presented as follows.
 There is one generator for each bounded acyclic binary multicomplex of dimension $n$, and there are relations
 $[N'] + [N''] = [N]$ if there is a short exact sequence $\exact{N' }{N}{N''}$ in $(B^{\tn{q}})^{n}\mc{N}$, 
 and $[T] = 0$ if $T$ is a diagonal acyclic binary multicomplex.
\end{thmdef}

Observe that if we disregard the second relation that diagonal binary multicomplexes vanish, then
we obtain the Grothendieck group $K_{0}(B^{\tn{q}})^{n}\mc{N}$.
Another way to say this is that $K_{n}\mc{N}$ is a quotient group of the Grothendieck group
of the exact category $(B^{\tn{q}})^{n}\mc{N}$.
Denote by $T^{n}_{\mc{N}}$ the subgroup of $K_{0}(B^{\tn{q}})^{n}\mc{N}$ generated
by the classes of the diagonal binary multicomplexes in $(B^{\tn{q}})^{n}\mc{N}$.
Then we may write $K_{n}\mc{N} \cong (K_{0}(B^{\tn{q}})^{n}\mc{N}) / T^{n}_{\mc{N}}$. \\

We now present a new, elementary proof of the additivity theorem. 
Let $\mc{A}$ and $\mc{C}$ be exact subcategories of an exact category $\mc{B}$. 

\begin{defn}
\label{extension category}
 The \textit{extension category} $\mc{E}(\mc{A},\mc{B},\mc{C})$ 
 is the exact category whose objects are short exact sequences
 $\exact{A}{B}{C}$ in $\mc{B}$, with $A$ in $\mc{A}$ and $C$ in $\mc{C}$, and 
 whose morphisms are commuting rectangles.
\end{defn}

\begin{thm}[Additivity]
\label{additivity}
 $K_{n}\mc{E}(\mc{A},\mc{B},\mc{C}) \cong K_{n}\mc{A} \times K_{n}\mc{C}$, for every $n \geq 0$.
\end{thm}

\begin{proof}
 The exact functors 
 \[
  \begin{array}{rcl}
  \mc{E}(\mc{A},\mc{B},\mc{C}) &\rightarrow& \mc{A} \times \mc{C} \\
  (\exact{A}{B}{C}) &\mapsto& (A,C) 
  \end{array}
 \]
 and
 \[
  \begin{array}{rcl}
  \mc{A} \times \mc{C} &\rightarrow& \mc{E}(\mc{A},\mc{B},\mc{C})\\
  (A,C) &\mapsto& (\exact{A}{A \oplus C}{C})
  \end{array}
 \]
 induce mutually inverse isomorphisms between $K_{0}\mc{E}(\mc{A},\mc{B},\mc{C})$ and 
 $K_{0}\mc{A} \times K_{0}\mc{C}$, (see, e.g., \cite{WeiK}, II.9.3).
 For each $n > 0$, the categories $(B^{\tn{q}})^{n}\mc{A}$ and $(B^{\tn{q}})^{	n}\mc{C}$ are exact 
 subcategories of $(B^{\tn{q}})^{n}\mc{B}$, so we can define an $n^{\tn{th}}$ extension category
 \[ \mc{E}^{n}(\mc{A},\mc{B},\mc{C}) := 
 \mc{E}((B^{\tn{q}})^{n}\mc{A},(B^{\tn{q}})^{n}\mc{B},(B^{\tn{q}})^{n}\mc{C}). \]
 From the above, for each $n$ the induced map
 \[ K_{0}\mc{E}^{n}(\mc{A},\mc{B},\mc{C})
 \rightarrow K_{0}(B^{\tn{q}})^{n}\mc{A} \times K_{0}(B^{\tn{q}})^{n}\mc{C} \]
 is an isomorphism. But a short exact sequence of binary multicomplexes is exactly the same thing as
 a binary multicomplex of short exact sequences, so the categories $\mc{E}^{n}(\mc{A},\mc{B},\mc{C})$ and
 $(B^{\tn{q}})^{n}\mc{E}(\mc{A},\mc{B},\mc{C})$ are isomorphic, so we therefore have an isomorphism
 \[
 K_{0}(B^{\tn{q}})^{n}\mc{E}(\mc{A},\mc{B},\mc{C}) \cong K_{0}(B^{\tn{q}})^{n}\mc{A} \times K_{0}(B^{\tn{q}})^{n}\mc{C} 
 \]
 for each $n$.
 Identifying the categories $\mc{E}^{n}(\mc{A},\mc{B},\mc{C})$ and $(B^{\tn{q}})^{n}\mc{E}(\mc{A},\mc{B},\mc{C})$, 
 a binary multicomplex in 
 $(B^{\tn{q}})^{n}\mc{E}(\mc{A},\mc{B},\mc{C})$ is diagonal in direction $i$ if and only if its 
 constituent binary multicomplexes in $(B^{\tn{q}})^{n}\mc{A}, (B^{\tn{q}})^{n}\mc{B}$ and $ (B^{\tn{q}})^{n}\mc{C}$
 are also diagonal in direction $i$. Similarly, if $A$ in $(B^{\tn{q}})^{n}\mc{A}$ and $C$ in 
 $(B^{\tn{q}})^{n}\mc{C}$ are diagonal then the binary multicomplexes corresponding to $(\exact{A}{A}{0})$
 and $(\exact{0}{C}{C})$ are diagonal as well, so the isomorphism
 $K_{0}(B^{\tn{q}})^{n}\mc{E}(\mc{A},\mc{B},\mc{C})
 \cong K_{0}(B^{\tn{q}})^{n}\mc{A} \times K_{0}(B^{\tn{q}})^{n}\mc{C}$ restricts to an isomorphism
 $T^{n}_{\mc{E}(\mc{A},\mc{B},\mc{C})} \cong  T^{n}_{\mc{A}} \times T^{n}_{\mc{B}}$. 
 Passing to the quotients yields the result.
\end{proof}

\section{The resolution theorem}
\label{resolution}
The resolution theorem relates the $K$-theory of an exact category to that of a larger exact category all of whose objects
have a finite resolution by objects of the first category. Its most well-known application states that the $K$-theory of a 
regular ring is isomorphic to its so-called $G$-theory (the $K$-theory of the exact category of \textit{all} 
finitely-generated $R$-modules). As in the proof of the additivity theorem, we adapt a simple proof for 
$K_{0}$ to work for all $K_{n}$. The main difficulty in this proof is verifying that the hypotheses of the theorem pass 
to exact categories of acyclic binary multicomplexes.\\

The general resolution theorem for exact categories (\cite{Quil1}, $\S 4$ Corollary 2) is a 
formal consequence of the following theorem, which is Theorem 3 of \cite{Quil1}.

\begin{thm}
\label{theorem 3}
 Let $\mc{P}$ be a full, additive subcategory of an exact category $\mc{M}$ that
 is closed under extensions and satisfies:
\begin{enumerate}
 \item For any exact sequence $\exact{P'}{P}{M}$ in $\mc{M}$, if $P$ is in $\mc{P}$
 then $P'$ is in $\mc{P}$.
 \item For any $M$ in $\mc{M}$ there exists a $P$ in $\mc{P}$ and an admissible
 epimorphism $P \twoheadrightarrow M$.
\end{enumerate}
 Then the inclusion functor $\mc{P} \hookrightarrow \mc{M}$ induces an isomorphism
 $K_{n}\mc{P} \cong K_{n}\mc{M}$ for all $n \geq 0$.
\end{thm}

\begin{proof}
For $K_{0}$ the inverse
to the induced homomorphism $K_{0}\mc{P} \rightarrow K_{0}\mc{M}$ is given by the map 
\[
\begin{array}{rrcl}
 \phi : & K_{0}\mc{M} & \rightarrow & K_{0}\mc{P} \\
 & [M] & \mapsto & [P] - [P'],
\end{array}
\]
where $\exact{P'}{P}{M}$ is a short exact sequence of $\mc{M}$. 
The proof of Theorem\autoref{theorem 3} for $n=0$ is the simple exercise of checking that $\phi$ is well-defined.
We noted earlier that if $\mc{P}$ is closed under extensions in $\mc{M}$ then $(B^{\tn{q}})^{n}\mc{P}$ is closed 
under extensions in $(B^{\tn{q}})^{n}\mc{M}$ for each $n$, and by the same reasoning, one easily sees that if 
$\mc{P}$ and $\mc{M}$ satisfy hypothesis $(1)$ of the theorem, then so do $(B^{\tn{q}})^{n}\mc{P}$ and 
$(B^{\tn{q}})^{n}\mc{M}$ for each $n$. 
The following proposition is about hypothesis $(2)$.
\begin{prop}
\label{binary-multi-epi}
 Let $\mc{P}$ and $\mc{M}$ satisfy the hypotheses of\autoref{theorem 3}.
 For every object $M$ of $(B^{\tn{q}})^{n}\mc{M}$ there exists a short exact sequence
 $\exact{P'}{P}{M}$ of $(B^{\tn{q}})^{n}\mc{M}$ with $P'$ and $P$ in $(B^{\tn{q}})^{n}\mc{P}$. 
 Furthermore, if $M$ is a diagonal binary multicomplex then we may choose
 $P$ and $P'$ to be diagonal as well.
\end{prop}
We shall prove Proposition\autoref{binary-multi-epi} shortly. We now continue with the proof of
Theorem\autoref{theorem 3}. Together with the known isomorphism for $K_{0}$, 
the first part of the proposition implies that the induced map
$K_{0}(B^{\tn{q}})^{n}\mc{P} \rightarrow K_{0}(B^{\tn{q}})^{n}\mc{M}$ is an isomorphism
for each $n$. Clearly this isomorphism sends elements of $T^{n}_{\mc{P}}$ to elements of $T^{n}_{\mc{M}}$.
Since the value of $\phi$ is independent of the choice of resolution, the second part of the proposition
implies that $\phi$ maps elements of $T^{n}_{\mc{M}}$ to elements of $T^{n}_{\mc{P}}$. The isomorphism
therefore descends to an isomorphism $K_{n}\mc{P} \rightarrow K_{n}\mc{M}$.
\end{proof}

It remains then to prove Proposition\autoref{binary-multi-epi}, 
so for the rest of this section we assume the hypotheses of Theorem\autoref{theorem 3}. 
The idea of the proof is to construct, for each $M$ in $(B^{\tn{q}})^{n}\mc{M}$,
a morphism of acyclic binary chain complexes $P \rightarrow M$ that is term-wise and admissible epimorphism, i.e. 
$P_{j} \twoheadrightarrow M_{j}$. By the assumption on $\mc{P}$ and $\mc{M}$ each of these admissible epimorphisms
is part of a short exact sequence $\exact{P'_{j}}{P_{j}}{M_{j}}$ with the $P_{j}$ in $\mc{P}$. 
The $P'_{j}$ form a binary complex with the induced maps, and we show that this binary complex is in fact acyclic. 
The result will then follow from an induction on the dimension. We shall rely on the following fact.

\begin{lem}
 \label{admissible sums}
 Let $f_{i}: Q_{i} \rightarrow N$, $i = 1, \dots ,m$ be a family of morphisms in an exact category,
 at least one of which is an admissible epimorphism. Then the induced morphism
 \[
 \left[\begin{smallmatrix} f_{1} & \dots & f_{m} \end{smallmatrix}\right]: \bigoplus_{i=1}^{m} Q_{i} \rightarrow N
 \]
 is an admissible epimorphism as well. \qed
\end{lem}

\begin{proof}
 The general case follows from the case $m=2$. In this case, the morphism 
 $\left[\begin{smallmatrix} f_{1} & f_{2} \end{smallmatrix}\right]$ factors as the composition
 \[
 Q_{1} \oplus Q_{2} \stackrel{\left[\begin{smallmatrix} f_{1} & 0 \\ 0 & 1  \end{smallmatrix}\right]}{\longrightarrow} 
 N \oplus Q_{2} \stackrel{\left[\begin{smallmatrix} 1 & f_{2} \\ 0 & 1  \end{smallmatrix}\right]}{\longrightarrow}
 N \oplus Q_{2} \stackrel{\left[\begin{smallmatrix} 1 & 0 \end{smallmatrix}\right]}{\longrightarrow} N,
 \]
 all of which are admissible epimorphisms.
\end{proof}

We begin resolving binary complexes in the less involved case, in which we assume $M$ to be diagonal. 

\begin{lem}
\label{diagonal res}
 Given a diagonal bounded acyclic binary complex $M$ in $B^{\tn{q}}\mc{M}$ there exists a
 short exact sequence $\exact{P'}{P}{M}$ where $P'$ and $P$ are diagonal objects of $B^{\tn{q}}\mc{P}$.
\end{lem}

\begin{proof}
 We may consider $M$ as an object of $C^{\tn{q}}\mc{N}$, as 
 $\Delta : C^{\tn{q}}\mc{N} \rightarrow B^{\tn{q}}\mc{N}$ is a full embedding for any exact category $\mc{N}$.
 Represent $M$ in $C^{\tn{q}}\mc{N}$ as below. Without loss of generality we assume that
 $M$ ends at place $0$.
 \[ 
 \xymatrix{
 0 \ar[r] 
 & M_{n} \ar^{d}[r] 
 & \cdots \ar^{d}[r]  
 & M_{k} \ar^{d}[r] 
 & M_{k - 1} \ar^{d}[r] 
 & \cdots \ar^{d}[r] 
 & M_{0} \ar[r] 
 & 0, 
 }
 \]
 Since $\mc{P}$ and $\mc{M}$ satisfy the hypotheses of Theorem\autoref{theorem 3}, 
 there exists an object $Q_{k}$ of $\mc{P}$ and an admissible epimorphism
 $\epsilon_{k} : Q_{k} \twoheadrightarrow M_{k}$ in $\mc{M}$ for each $0 \leq k \leq n$.
 The diagram below is a morphism in $B^{\tn{q}}\mc{M}$ with target $M$,
 and its upper row, the source of the morphism, is an object of $C^{\tn{q}}\mc{P}$.
\[ 
\xymatrix{
0 \ar[r] \ar[d] 
& 0 \ar[d]
\ar[r] 
& \cdots \ar[r] 
& Q_{k} \ar@{=}[r]  \ar@{->>}^{\epsilon_{k}}[d]
& Q_{k} \ar[r] \ar^{d\epsilon_{k}}[d]
& \cdots \ar[r] 
& Q \ar[r]  \ar[d]
& Q \ar[d] \\
0 \ar[r] 
& M_{n} \ar^{d}[r] 
& \cdots \ar^{d}[r]
& M_{k} \ar^{d}[r]
& M_{k - 1} \ar^{d}[r]
& \cdots \ar^{d}[r] 
& M_{0} \ar[r] 
& 0  
 } 
 \]
We denote the top row by $P^{k}$ and the morphism of complexes by $\zeta^{k} : P^{k} \rightarrow M$.
We do this for each $k \in \{ 0, \dots , n \}$ and form the sum
\[
 \zeta := \left[\begin{smallmatrix}  \zeta^{n} & \dots & \zeta^{0} \end{smallmatrix}\right]
 : \bigoplus P^{k} \rightarrow M.
 \]
 Call the direct sum $P$. Each of the complexes $P^{k}$ is acyclic, so $\mc{P}$ is as well. Consider the
 morphism $\zeta_{j} : P_{j} \rightarrow M_{j}$, the part of $\zeta$ from the $j^{\tn{th}}$ term of $P_{j}$ to the
 $j^{\tn{th}}$ term of $M$. By construction and Lemma\autoref{admissible sums}, it is an admissible 
 epimorphism of $\mc{M}$.
 We now have a morphism $\zeta : P \rightarrow M$ that has term-wise morphisms all
 admissible epimorphisms. The kernels of the term-wise admissible epimorphisms are all in $\mc{P}$ 
 by the hypotheses on $\mc{P}$ and $\mc{M}$.
 These kernels form a (as we have seen in $\S 1$, not \textit{a priori} acyclic) bounded chain complex 
 $P'$ in $\mc{P}$ under the induced maps between them. Moreover, $\mc{P}'$ must be diagonal as $P$ and $M$ are. \\
 It remains to show that this 
 chain complex is acyclic (i.e., that $P'$ is in $B^{\tn{q}}\mc{P}$),
 then $\exact{P'}{P}{M}$ will be a short exact sequence of acyclic complexes. 
 The objects of $P'$ are in
 $\mc{P}$ and $P'$ is acyclic in $\mc{M}$ by standard homological algebra, 
 but this does not guarantee the acyclicity of $P'$ in $\mc{P}$, as evidenced by 
 Example\autoref{complex-epi}. Suppose that $M$ factors through objects $Z_{j}$ of $\mc{M}$, and
 that each $P^{k}$ factors through objects $Y^{k}_{j}$ of $\mc{P}$. Then $P'$ factors through the kernels
 of the corresponding morphisms 
 \[
  \bigoplus Y^{k}_{j} \rightarrow Z_{j}.
 \]
 Since $\mc{P}$ is closed under kernels of admissible epimorphisms to objects of $\mc{M}$, 
 it is enough to show that each of these morphisms is an admissible epimorphism,
 and by Lemma\autoref{admissible sums} it is enough in turn to show that, for each $j$, one of the morphisms
 $Y^{k}_{j} \rightarrow Z_{j}$ is an admissible epimorphism. The diagram below shows that this
 is true for $k = j$, as $Y^{j}_{j} = Q_{j}$.
 \[
  \xymatrix{
  Q_{j} \ar@{=}[r] \ar@{->>}[d] & Q_{j} \ar@{=}[r] \ar[d] & Q_{j} \ar[d] \\
  M_{j} \ar@{->>}[r] & Z_{j} \ar@{ >->}[r] & M_{j-1}.
  }
 \]
 Finally we consider $P'$, $P$ and $M$ now as diagonal binary complexes (by applying $\Delta$). Then
 $\exact{P'}{P}{M}$ is the required short exact sequence of acyclic diagonal binary complexes.
\end{proof}

A little more work is required if the binary complex $M$ is not diagonal. The idea in this case is due to Grayson,
and relies on the acyclicity of the chain complexes
\[ 0 \longrightarrow Q 
  \stackrel{\left[\begin{smallmatrix} 1 \\ 0 \end{smallmatrix}\right]}{\longrightarrow} Q \oplus Q
  \stackrel{\left[\begin{smallmatrix} 0 & 1 \\ 0 & 0 \end{smallmatrix}\right]}{\longrightarrow} Q \oplus Q
  \stackrel{\left[\begin{smallmatrix} 0 & 1 \\ 0 & 0 \end{smallmatrix}\right]}{\longrightarrow} \cdots
  \stackrel{\left[\begin{smallmatrix} 0 & 1 \\ 0 & 0 \end{smallmatrix}\right]}{\longrightarrow} Q \oplus Q
  \stackrel{\left[\begin{smallmatrix} 0 & 1  \end{smallmatrix}\right]}{\longrightarrow} Q
  \longrightarrow 0
  \]
and
  \[ 0 \longrightarrow Q 
  \stackrel{\left[\begin{smallmatrix} 0 \\ 1 \end{smallmatrix}\right]}{\longrightarrow} Q \oplus Q
  \stackrel{\left[\begin{smallmatrix} 0 & 0 \\ 1 & 0 \end{smallmatrix}\right]}{\longrightarrow} Q \oplus Q
  \stackrel{\left[\begin{smallmatrix} 0 & 0 \\ 1 & 0 \end{smallmatrix}\right]}{\longrightarrow} \cdots
  \stackrel{\left[\begin{smallmatrix} 0 & 0 \\ 1 & 0 \end{smallmatrix}\right]}{\longrightarrow} Q \oplus Q
  \stackrel{\left[\begin{smallmatrix} 1 & 0  \end{smallmatrix}\right]}{\longrightarrow} Q
  \longrightarrow 0
  \]
of arbitrary length, where $Q$ is an object of any exact category.

\begin{lem}
\label{arbitrary res}
 Given an arbitrary bounded acyclic binary complex $M$ in $B^{\tn{q}}\mc{M}$ there exists a
 short exact sequence $\exact{P'}{P}{M}$, where $P'$ and $P$ are objects of $B^{\tn{q}}\mc{P}$.
\end{lem}

\begin{proof}
 Let $M$ denote the element of $B^{q}\mc{M}$ given by the binary complex below
\[ \xymatrix{
0 \ar@<.5ex>[r] \ar@<-.5ex>[r] 
& M_{n} \ar@<.5ex>^{d}[r] \ar@<-.5ex>_{d'}[r]
& \cdots \ar@<.5ex>^{d}[r] \ar@<-.5ex>_{d'}[r] 
& M_{k} \ar@<.5ex>^{d}[r] \ar@<-.5ex>_{d'}[r] 
& M_{k-1} \ar@<.5ex>^{d}[r] \ar@<-.5ex>_{d'}[r] 
& \cdots \ar@<.5ex>^{d}[r] \ar@<-.5ex>_{d'}[r] 
& M_{0} \ar@<.5ex>[r] \ar@<-.5ex>[r] 
& 0 
} \]
and as before let $\epsilon_{k} : Q_{k} \twoheadrightarrow M_{k}$
be admissible epimorphisms in $\mc{M}$ with $Q_{k}$ in $\mc{P}$. 
Inductively define two collections of morphisms $\delta_{k, l}, \delta'_{k,l} : Q_{k} \rightarrow M_{k-l}$
by 
\begin{equation*}
  \begin{cases}
    \delta_{k,1} = d \circ \epsilon_{k} \\
    \delta'_{k,1} = d' \circ \epsilon_{k} 
  \end{cases}
\end{equation*}
and
\begin{equation*}
  \begin{cases}
    \delta_{k,l+1} = d \circ \delta'_{k,l} \\
    \delta'_{k,l+1} = d' \circ \delta_{k,l}. 
  \end{cases}
\end{equation*}
Since each of its differentials is acyclic, the top row of the diagram below is an object of
$B^{\tn{q}}\mc{P}$ for each $k \in \{ 0, \dots , n \}$.
\[ \xymatrix{
0 \ar@<.5ex>[r] \ar@<-.5ex>[r]
& Q_{k} \ar@{->>}^{\epsilon_{k}}[dd] 
\ar@<.5ex>^{\left[\begin{smallmatrix} 1 \\ 0 \end{smallmatrix}\right]}[r]
\ar@<-.5ex>_{\left[\begin{smallmatrix} 0 \\ 1 \end{smallmatrix}\right]}[r]
& Q_{k} \oplus Q_{k} \ar^{\left[\begin{smallmatrix} \delta_{k,1} & \delta'_{k,1} \end{smallmatrix}\right]}[dd] 
\ar@<.5ex>^{\left[\begin{smallmatrix} 0 & 1 \\ 0 & 0 \end{smallmatrix}\right]}[r]
\ar@<-.5ex>_{\left[\begin{smallmatrix} 0 & 0 \\ 1 & 0 \end{smallmatrix}\right]}[r]
& Q_{k} \oplus Q_{k} \ar^{\left[\begin{smallmatrix} \delta_{k,2} & \delta'_{k,2} \end{smallmatrix}\right]}[dd] 
\ar@<.5ex>^{\left[\begin{smallmatrix} 0 & 1 \\ 0 & 0 \end{smallmatrix}\right]}[r]
\ar@<-.5ex>_{\left[\begin{smallmatrix} 0 & 0 \\ 1 & 0 \end{smallmatrix}\right]}[r]
& \cdots \ar@<.5ex>^{\left[\begin{smallmatrix} 0 & 1 \\ 0 & 0 \end{smallmatrix}\right]}[r]
\ar@<-.5ex>_{\left[\begin{smallmatrix} 0 & 0 \\ 1 & 0 \end{smallmatrix}\right]}[r]
& Q_{k} \oplus Q_{k} \ar^{\left[\begin{smallmatrix} \delta_{k,k} & \delta'_{k,k} \end{smallmatrix}\right]}[dd] 
\ar@<.5ex>^{\left[\begin{smallmatrix} 1 & 0  \end{smallmatrix}\right]}[r]
\ar@<-.5ex>_{\left[\begin{smallmatrix} 0 & 1  \end{smallmatrix}\right]}[r]
& Q_{k} \ar[dd]  \\
\\
\cdots \ar@<.5ex>[r] \ar@<-.5ex>[r]
& M_{k} \ar@<.5ex>^{d}[r] \ar@<-.5ex>_{d'}[r] 
& M_{k-1} \ar@<.5ex>^{d}[r] \ar@<-.5ex>_{d'}[r] 
& M_{k-2} \ar@<.5ex>^{d}[r] \ar@<-.5ex>_{d'}[r] 
& \cdots \ar@<.5ex>^{d}[r] \ar@<-.5ex>_{d'}[r] 
& M_{0} \ar@<.5ex>[r] \ar@<-.5ex>[r] 
& 0.
} \]
The morphisms $\delta_{k, l}$ and $\delta'_{k,l}$ have been 
constructed so that the vertical morphisms commute with the top and bottom differentials, so the diagram
represents a morphism in $B^{\tn{q}}\mc{M}$, which we shall again denote by $\zeta^{k} : P^{k} \rightarrow M$. 
Following the same method of proof as of the previous lemma, each $P^{k}$ is acyclic so their direct sum 
is acyclic as well and so
\[
 \zeta := \left[\begin{smallmatrix} \zeta^{n} & \dots & \zeta^{0} \end{smallmatrix}\right]
 : \bigoplus P^{k}  \rightarrow M.
 \]
 is a morphism in $B^{\tn{q}}\mc{M}$.
 By construction and Lemma\autoref{admissible sums} again, each term-wise morphism 
 $\zeta^{j} : P_{j} \rightarrow M_{j}$ is an
 admissible epimorphism in $\mc{M}$. Each of these morphisms therefore has a kernel in $\mc{P}$ and 
 these kernels form a binary complex with the induced maps. \\
 We wish to show that both differentials of this binary complex is acyclic in $\mc{P}$. 
 Consider the top differential first. Denoting the objects that the top differentials of $M$ and each $P^{k}$ 
 factor through by $Z_{j}$ and $Y^{k}_{j}$, it is enough to
 show for each $j$ that one of the morphisms $Y^{k}_{j} \rightarrow Z_{j}$ is an admissible epimorphism,
 exactly as in the proof of Lemma\autoref{diagonal res}.
 Taking $k = j$ again yields the result, as shown by the diagram below.
 \[
  \xymatrix{
  Q_{j}  \ar@{=}[r] \ar@{->>}[d] & Q_{j} \ar@{ >->}[r]^{\left[\begin{smallmatrix} 1 \\ 0 \end{smallmatrix}\right]}
  \ar[d] & Q_{j} \oplus Q_{j} \ar[d] \\
  M_{j} \ar@{->>}[r] & Z_{j} \ar@{ >->}[r] & M_{j-1}.
  }
 \]
 The bottom differential is dealt with entirely analogously.
\end{proof}

\begin{proof}[Proof of Proposition\autoref{binary-multi-epi}.]
 We proceed by induction on $n$. In the base case $n = 0$, there is nothing to show. 
 For the inductive step, we view an acyclic binary multicomplex $M$ in $(B^{\tn{q}})^{n+1}\mc{M}$,
 as an acyclic binary multicomplex of objects of $(B^{\tn{q}})^{n}\mc{M}$, i.e. as an object of
 $B^{q}(B^{\tn{q}})^{n}\mc{M}$. By the inductive hypothesis, the inclusion of $(B^{\tn{q}})^{n}\mc{P}$ into
 $(B^{\tn{q}})^{n}\mc{M}$ satisfies the hypotheses of Theorem\autoref{theorem 3}, so by Lemma\autoref{arbitrary res}
 there exists a short exact sequence $\exact{P'}{P}{M}$ in $B^{q}(B^{\tn{q}})^{n}\mc{M}$ with $P'$ and $P$ in
 $B^{q}(B^{\tn{q}})^{n}\mc{P} = (B^{\tn{q}})^{n+1}\mc{P}$, and so the first part follows. For the second part,
 suppose that $M$ is diagonal in some direction $i$. We consider
 $M$ as a diagonal acyclic binary complex of (not necessarily diagonal) objects of $(B^{\tn{q}})^{n}\mc{M}$,
 that is, we ``expand'' $M$ along the $i$ direction.
 Then by Lemma\autoref{diagonal res} there exist diagonal acyclic binary complexes $P'$ and $P$ in 
 $(B^{\tn{q}})^{n+1}\mc{P}$ that are diagonal in direction $i$, and an exact sequence
 $\exact{P'}{P}{M}$, so the proof is complete.
\end{proof}

\section{The cofinality theorem}
\label{cofinality}
Unlike the additivity and resolution theorems, the cofinality theorem was not proved by Quillen in \cite{Quil1}.
A proof for exact categories based on work by Gersten was given in \cite{Gray3}. More general versions can be found in
\cite{Staffeldt} and \cite{Thomason}. It is proven in \cite{Gray1} that the hypotheses of the cofinality theorem 
are satisfied by the appropriate exact categories of acyclic binary complexes, the main work in our proof is in 
ensuring that the results pass to the quotients defining $K_{n}$.

\begin{defn}
 An exact subcategory $\mc{M}$ of an exact category $\mc{N}$ is said to be \textit{cofinal} in $\mc{N}$ if
 for every object $N_{1}$ of $\mc{N}$ there exists another object $N_{2}$ of $\mc{N}$ such that
 $N_{1} \oplus N_{2}$ is isomorphic to an object of $\mc{M}$.
\end{defn}

An obvious example of a cofinal exact subcategory is the category of free $R$-modules inside the category 
of projective $R$-modules, for any ring $R$. More generally, every exact category is cofinal in
its \textit{Karoubification} (\cite{Thomason}, Appendix A). The cofinality theorem 
relates the $K$-theory of the cofinal subcategory to the $K$-theory of the exact category containing it.
Throughout this section $\mc{M} \hookrightarrow \mc{N}$ is the inclusion of a cofinal exact subcategory 
of an exact category $\mc{N}$.\\

Define an equivalence relation on the objects of $\mc{N}$ by declaring $N_{1} \sim N_{2}$ if there exist objects
$M_{1}$ and $M_{2}$ of $\mc{M}$ such that 
\[ N_{1} \oplus M_{1} \cong N_{2} \oplus M_{2}. \]
Since $\langle M \rangle = 0$ for every $M$ in $\mc{M}$, the cofinality
of $\mc{M}$ in $\mc{N}$ ensures that equivalence classes of $\sim$ form a group under the natural operation
$\langle N_{1} \rangle + \langle N_{2} \rangle = \langle N_{1} \oplus N_{2} \rangle$; 
we denote this group by $K_{0}(\mc{N} \ \rm{rel.} \ \mc{M} )$. The following lemma and its corollary 
were first observed in the proof of Theorem 1.1 in \cite{Gray3}.

\begin{lem}
\label{K0 inj}
 The sequence:
 \[
 \begin{array}{rcccccccl}
  0 \longrightarrow & K_{0}\mc{M} & \longrightarrow & K_{0}\mc{N} & \longrightarrow & 
  K_{0}(\mc{N} \ \rm{rel.} \ \mc{M} ) & \longrightarrow & 0 \\
  & & & [N] & \longmapsto & \langle N \rangle 
 \end{array}
 \]
 is well-defined and exact. \qed
\end{lem}

\begin{cor}
 \label{cofinal pairs}
 For any pair of objects $N_{1}$, $N_{2}$ of $\mc{N}$ with the same class in 
 $K_{0}\mc{N} / K_{0}\mc{M}$ there exists a (single) object 
 $N'$ in $\mc{N}$ such that each $N_{i} \oplus N'$ is in $\mc{M}$.
\end{cor}

\begin{proof}
 By the lemma, if $N_{1}$ and $N_{2}$ have the same class in $K_{0}\mc{N} / K_{0}\mc{M}$ then
 $\langle N_{1} \rangle = \langle N_{2} \rangle$. From cofinality there exists a $P$ in $\mc{N}$ such that
 $N_{1} \oplus P$ is in $\mc{M}$, so 
 $\langle 0 \rangle = \langle N_{1} \oplus P \rangle = \langle N_{2} \oplus P \rangle$. Hence there exist objects
 $P_{1}$ and $P_{2}$ of $\mc{M}$ such that each $(N_{i} \oplus P) \oplus P_{i}$ is an object of $\mc{M}$.
 Setting $N' = P \oplus P_{1} \oplus P_{2}$, each $N_{i} \oplus N'$ is an object of $\mc{M}$.
\end{proof}

We show now that cofinality of $\mc{M}$ in $\mc{N}$ passes to the associated categories of acyclic binary multicomplexes. 
We can say much more in general however.

\begin{lem}
\label{cofinal directions}
 For all $n \geq 0$, the exact subcategory $(B^{\tn{q}})^{n}\mc{M}$ is cofinal in $(B^{\tn{q}})^{n}\mc{N}$.
 More precisely, if $N$ is in $(B^{\tn{q}})^{n}\mc{N}$ and $i \in \{ 1, \dots, n \}$ is any direction then
 there exists an object $T$ in $(B^{\tn{q}})^{n}\mc{N}$ that is diagonal in direction $i$ such that 
 $N \oplus T$ is in $(B^{\tn{q}})^{n}\mc{M}$. Moreover, if $N$ is diagonal in direction 
 $j \in \{ 1, \dots, n \}$, $j \neq i$, then $T$ may be chosen to be diagonal in directions $i$ \emph{and} $j$.
\end{lem}

\begin{proof}
 (Part of the following proof is adapted from the proof of Lemma 6.2 in \cite{Gray1}.) 
 We proceed by induction on $n$. 
 The statements for the base case $n = 0$ mean only that $\mc{M}$ is cofinal in $\mc{N}$, which is assumed 
 throughout this section.\\
 For the inductive step we fix $i \in \{ 1 , \dots , n + 1\}$ and $N$ in $(B^{\tn{q}})^{n+1}\mc{N}$. We first assume
 that $N$ is diagonal in direction $j \neq i$, and ``expand along $j$'' to consider $N$ as a diagonal acyclic binary 
 complex of objects of $(B^{\tn{q}})^{n}\mc{N}$. Let $C_{k}$ in $(B^{\tn{q}})^{n}\mc{N}$ be the image of
 $d_{k}^{j} = \tilde{d}_{k}^{j}$ (between the terms $k$ and $k-1$). By the inductive hypothesis, for each $k$ 
 there exists an object $T_{k}$ in $(B^{\tn{q}})^{n}\mc{N}$ that is diagonal in direction $i$ such that 
 $C_{k} \oplus T_{k}$ is an object of $(B^{\tn{q}})^{n}\mc{M}$. Let $(T', e)$ be the acyclic chain complex
 of objects of $(B^{\tn{q}})^{n}\mc{N}$ given by taking the direct sum of the identity maps $T_{k} \rightarrow T_{k}$ 
 concentrated in degrees $k$ and $k-1$. That is, $(T', e)$ is the complex
 \[ 
  0 \rightarrow  \cdots \rightarrow T_{k+2} \oplus T_{k+1} \rightarrow T_{k+1} \oplus T_{k}
  \rightarrow T_{k} \oplus T_{k-1} \rightarrow \cdots \rightarrow 0 
 \]
with differential 
$e = \bigl(\begin{smallmatrix}
 0&1\\ 0&0
 \end{smallmatrix} \bigr)$. Since each $T_{k}$ is diagonal in direction $i$, the complex $T'$ is diagonal in direction $i$, 
 when regarded as an object of $C^{\tn{q}}(B^{\tn{q}})^{n}\mc{N}$.
 The image of the differential $d_{k}^{j} \oplus e_{k} = \tilde{d}_{k}^{j} \oplus e_{k}$ on the acyclic complex
 $\bot_{j}(N) \oplus T'$ is equal to $C_{k} \oplus T_{k}$, and hence belongs to $(B^{\tn{q}})^{n}\mc{M}$.
 Since $(B^{\tn{q}})^{n}\mc{M}$ is closed under extensions in $(B^{\tn{q}})^{n+1}\mc{M}$, we obtain that
 the complex $\bot_{j}(N) \oplus T'$ belongs to $C^{\tn{q}}(B^{\tn{q}})^{n}\mc{M}$. We define a binary complex
 $T : = \Delta_{j}(T')$, which is an object of $B^{\tn{q}}(B^{\tn{q}})^{n}\mc{N} = (B^{\tn{q}})^{n+1}\mc{N}$ and 
 is diagonal in directions $i$ \emph{and} $j$. Then $N \oplus T = \Delta_{j}(\bot_{j}(N) \oplus T')$ belongs to 
 $B^{\tn{q}}(B^{\tn{q}})^{n}\mc{M} = (B^{\tn{q}})^{n+1}\mc{M}$.\\
 If $N$ is not diagonal in any direction different from $i$, we again consider $N$ as an acyclic binary complex
 \[ \xymatrix{
\cdots \ar@<.5ex>^{d}[r] \ar@<-.5ex>_{d'}[r]
& N_{k+1} \ar@<.5ex>^{d}[r] \ar@<-.5ex>_{d'}[r] 
& N_{k} \ar@<.5ex>^{d}[r] \ar@<-.5ex>_{d'}[r] 
& N_{k-1} \ar@<.5ex>^{d}[r] \ar@<-.5ex>_{d'}[r] 
& \cdots, 
} \]
of objects $N_{k}$ in $(B^{\tn{q}})^{n}\mc{N}$, but now ``expanded along'' direction $i$. Let $C_{k}$
and $\widetilde{C}_{k}$ in $(B^{\tn{q}})^{n}\mc{N}$ denote the images of 
the (normally different) differentials $d_{k}^{i}$ and $\tilde{d}_{k}^{i}$ 
(between $k$ and $k-1$). The classes of both $C_{k}$ and $\widetilde{C}_{k}$ are equal to 
the finite sum $\sum_{i=-\infty}^{k}(-1)^{k-i}[N_{k-i}]$ in $K_{0}(B^{\tn{q}})^{n}\mc{N}$, 
by a standard argument, so they have the same class in
\[ \operatorname{coker}(K_{0}(B^{\tn{q}})^{n}\mc{M} \rightarrow K_{0}(B^{\tn{q}})^{n}\mc{N}). \]
By the inductive hypothesis and Corollary\autoref{cofinal pairs}, there exists a single  object $T_{k}$ in 
$(B^{\tn{q}})^{n}\mc{N}$ such that $C_{k} \oplus T_{k}$ and $\tilde{C}_{k} \oplus T_{k}$ are both objects
of $(B^{\tn{q}})^{n}\mc{N}$. As above, from the objects $T_{k}$ we form the acyclic binary complex $T$ in
$B^{\tn{q}}(B^{\tn{q}})^{n}\mc{N} = (B^{\tn{q}})^{n+1}\mc{N}$, which is diagonal in direction $i$. Then
$N \oplus T$ is an object of $(B^{\tn{q}})^{n+1}\mc{M}$, as was to be shown.
\end{proof}

Following Lemmas\autoref{K0 inj} and\autoref{cofinal directions}, we now regard $K_{0}(B^{\tn{q}})^{n}\mc{M}$
as a subgroup of $K_{0}(B^{\tn{q}})^{n}\mc{N}$ for each $n$.
It is clear moreover that this inclusion respects the subgroups generated by the classes of diagonal 
binary multicomplexes, i.e. $T^{n}_{\mc{M}} \subseteq T^{n}_{\mc{N}}$. The following proposition concerning
representations of elements of $T^{n}_{\mc{N}} / T^{n}_{\mc{M}}$ is key to our proof of the cofinality theorem.

\begin{prop}
\label{diagonal represent}
 Let $x + T^{n}_{\mc{M}}$ be a class in $T^{n}_{\mc{N}} / T^{n}_{\mc{M}}$, for $n \geq 1$. Then 
 $x + T^{n}_{\mc{M}} = [t] + T^{n}_{\mc{M}}$, where $[t]$ is the  
 class in $K_{0}(B^{\textnormal{q}})^{n}\mc{N}$ of a diagonal acyclic 
 binary multicomplex $t$ in $(B^{\tn{q}})^{n}\mc{N}$.
\end{prop}

\begin{proof}
 The idea is to take the class of a general element $x \in T^{n}_{\mc{N}}$ and transform it into a direct sum of 
 diagonal complexes that are all diagonal in the same direction, without altering the class of $x$ modulo 
 $T^{n}_{\mc{M}}$. To begin, we write 
 \[ x + T^{n}_{\mc{M}} = \sum_{j=1}^{n}([t'_{j}] - [t_{j}'']) + T^{n}_{\mc{M}}, \]
 where each $t'_{j}$ and $t''_{j}$ is an actual acyclic binary multicomplex diagonal in direction $j$, and 
 pick a distinguished direction $i$. By Lemma\autoref{cofinal directions}, for each $j \neq i$ there exist
 acyclic binary multicomplexes $s'_{j}$ and $s''_{j}$ that are both diagonal in directions $i$ \emph{and} $j$ such that
 $t'_{j} \oplus s'_{j}$ and $t''_{j} \oplus s''_{j}$ are objects of $(B^{\textnormal{q}})^{n}\mc{M}$. The binary complexes
 $t'_{j}$, $s'_{j}$, $t''_{j}$ and $s''_{j}$ are all diagonal in direction $j$, so their direct sums are also diagonal in
 direction $j$, and we have
 $[t'_{j} \oplus s'_{j}] \in T^{n}_{\mc{M}}$ and $[t''_{j} \oplus s''_{j}] \in T^{n}_{\mc{M}}$. 
 Therefore $[s'_{j}] =-[t'_{j}]$ and $[s''_{j}] =-[t''_{j}]$ in $T^{n}_{\mc{N}} / T^{n}_{\mc{M}}$.
 But the $s'_{j}$ and $s''_{j}$, along with $t'_{i}$ and $t''_{i}$ are all diagonal in direction $i$,
 so taking $u_{1}$ to be the sum of the positive classes in our new expansion of $x + T^{n}_{\mc{M}}$,
 and $u_{2}$ to be the sum of the negative classes, we have
 \[ x + T^{n}_{\mc{M}} = [u_{1}] - [u_{2}] + T^{n}_{\mc{M}}, \]
 where $u_{1}$ and $u_{2}$ are acyclic binary multicomplexes that are diagonal in direction $i$.
 Finally, we use Lemma\autoref{cofinal directions} again to find an acyclic binary multicomplex
 $u'_{2}$ that is also diagonal in direction $i$, such that $[u_{2} \oplus u_{2}'] \in T^{n}_{\mc{M}}$.
 Then \[ x + T^{n}_{\mc{M}} = [u_{1}] + [u'_{2}] + T^{n}_{\mc{M}}, \]
 and setting $t = u_{1} \oplus u'_{2}$ yields the result.
\end{proof}

\begin{thm}[Cofinality]
\label{cofinality theorem}
 The inclusion functor 
 $\mc{M} \hookrightarrow \mc{N}$ induces an injection $K_{0}\mc{M} \hookrightarrow K_{0}\mc{N}$ and isomorphisms
 $K_{n}\mc{M} \cong K_{n}\mc{N}$ for $n > 0$.
\end{thm}

\begin{proof}
 The case $n = 0$ is part of Lemma\autoref{K0 inj}, so we proceed directly to $n > 0$.
 We have the following diagram of abelian groups, whose rows are exact.
 \[
  \xymatrix{
  0 \ar[r] & T^{n}_{\mc{M}} \ar[r] \ar@{^(->}[d] & K_{0}(B^{\tn{q}})^{n}\mc{M} \ar[r] \ar@{^(->}[d]
  & K_{n}\mc{M} \ar[r] \ar[d] & 0 \\
  0 \ar[r]& T^{n}_{\mc{N}} \ar[r] \ar@{->>}[d] & K_{0}(B^{\tn{q}})^{n}\mc{N} \ar[r] \ar@{->>}[d] 
  & K_{n}\mc{N} \ar[r] & 0 \\
  & T^{n}_{\mc{N}} / T^{n}_{\mc{M}} \ar[r] & K_{0}(B^{\tn{q}})^{n}\mc{N} / K_{0}(B^{\tn{q}})^{n}\mc{M} & &
  }
 \]
 The snake lemma implies that the homomorphism $K_{n}\mc{M} \rightarrow K_{n}\mc{N}$ is an isomorphism
 if and only if the induced homomorphism
 \[
 \begin{array}{rcl}
  T^{n}_{\mc{N}} / T^{n}_{\mc{M}} & \rightarrow & K_{0}(B^{\tn{q}})^{n}\mc{N} / K_{0}(B^{\tn{q}})^{n}\mc{M} \\
  x + T^{n}_{\mc{M}} & \mapsto & x + K_{0}(B^{\tn{q}})^{n}\mc{M}
 \end{array}                              
 \]
 is an isomorphism. Denote this homomorphism by $\psi$. We first show that $\psi$ is surjective. Let 
 $b$ be a generic element of $K_{0}(B^{\tn{q}})^{n}\mc{N}$, so $b = [b_{1}] - [b_{2}]$, where 
 $b_{1}$ and $b_{2}$ are acyclic binary multicomplexes of dimension $n$. By
 Lemma\autoref{cofinal directions} there exist diagonal acyclic binary multicomplexes $s_{1}$ and $s_{2}$ 
 such that $b_{i} \oplus s_{i}$ is an object of $(B^{\tn{q}})^{n}\mc{M}$ for $i = 1, 2$. Then 
 $[b_{1} \oplus s_{1}] - [b_{2} \oplus s_{2}] \in K_{0}(B^{\tn{q}})^{n}\mc{M}$, and is therefore zero in
 $K_{0}(B^{\tn{q}})^{n}\mc{N} / K_{0}(B^{\tn{q}})^{n}\mc{M}$. 
 Set $s = [s_{1}] - [s_{2}] \in T^{n}_{\mc{N}}$. Then $b + K_{0}(B^{\tn{q}})^{n}\mc{N}$
 is the image of $-s + T^{n}_{\mc{M}}$ under $\psi$, so $\psi$ is surjective.\\
 For the injectivity of $\psi$, suppose that $x \in T^{n}_{\mc{N}}$ such that $x + T^{n}_{\mc{M}}$ is in
 $\operatorname{ker}(\psi)$. By Proposition\autoref{diagonal represent} we may write 
 $x + T^{n}_{\mc{M}} = [t] + T^{n}_{\mc{M}}$ for an actual acyclic binary multicomplex $t$ diagonal 
 in some direction $i$. Since $[t] + T^{n}_{\mc{M}}$ is in the kernel of $\psi$, we must have
 $[t] \in K_{0}(B^{\tn{q}})^{n}\mc{M}$ (considered as a subgroup of $K_{0}(B^{\tn{q}})^{n}\mc{N}$). 
 In the notation of Lemma\autoref{K0 inj}, we have $\langle t \rangle = 0$ in 
 \[ K_{0}(B^{\textnormal{q}})^{n}\mc{N} / K_{0}(B^{\textnormal{q}})^{n}\mc{M} \cong 
 K_{0}((B^{\textnormal{q}})^{n}\mc{N} \ \text{rel.} \ (B^{\textnormal{q}})^{n}\mc{M}, \] 
 so there exist acyclic binary multicomplexes $a_{1}$ and $a_{2}$ in $(B^{\textnormal{q}})^{n}\mc{M}$
 such that $t \oplus a_{1} \cong a_{2}$. Consider the composite exact functor 
 $\Delta_{i}\top_{i} : (B^{\textnormal{q}})^{n}\mc{N} \rightarrow (B^{\textnormal{q}})^{n}\mc{N}$,
 that replaces the bottom differential in direction $i$ of an acyclic binary multicomplex
 with a second copy of the top differential. The binary multicomplexes $\Delta_{i}\top_{i}(a_{1})$
 and $\Delta_{i}\top_{i}(a_{2})$ are diagonal in direction $i$, and $\Delta_{i}\top_{i}(t) = t$, 
 since $t$ is already diagonal in direction $i$. Applying the induced homomorphism on $K_{0}$
 we have
 \[ [t]= K_{0} ( \Delta^{i} \top_{i})([t]) = K_{0} (\Delta_{i} \top_{i})([a_{1}] -[a_{2}]) 
 = [\Delta_{i} \top_{i}(a_{1})]- [\Delta_{i} \top_{i}(a_{2})]
  \in T^{n}_{\mc{M}}.
 \]
 Hence $x \in T^{n}_{\mc{M}}$ for any $x$ such that $x + T^{n}_{\mc{M}}$ is in the kernel of $\psi$, so the kernel
 of $\psi$ is trivial and $\psi$ is injective.
\end{proof}

\bibliographystyle{amsalpha}
\bibliography{biblio.bib}

\end{document}